\documentclass[11pt]{amsart}

\usepackage{amsfonts,amssymb,amsmath,amsthm,amscd,enumerate}
\usepackage{latexsym}
\usepackage{euscript}

\newtheorem{theorem}{Theorem}

\newtheorem{lemma}[theorem]{Lemma}
\newtheorem{corollary}[theorem]{Corollary}

\newtheorem{problem}{Problem}

\title[The subalgebra of graded central polynomials of an associative algebra]{The subalgebra of graded central polynomials \\ of an associative algebra}

\author{Galina Deryabina}

\address{Department of Computational Mathematics and Mathematical Physics (FS-11), Bauman Moscow State Technical University, 2-nd Baumanskaya Street, 5, 105005 Moscow, Russia}

\email{galina\_deryabina@mail.ru}

\author{Alexei Krasilnikov}

\address{Departamento de Matem\'atica, Universidade de Bras\'\i lia, 70910-900 Bras\'\i lia, DF, Brasil}

\email{alexei@unb.br}

\date{}

\begin{document}

\begin{abstract}
Let $F$ be a field and let $F \langle X \rangle$ be the free unital associative $F$-algebra on the free generating set $X = \{ x_1, x_2, \dots \}$. A subalgebra (a vector subspace) $V$ in $F \langle X \rangle$ is called a \textit{$T$-subalgebra} (a \textit{$T$-subspace}) if $\phi (V) \subseteq V$ for all endomorphisms $\phi$ of $F \langle X \rangle$. For an algebra $G$, its central polynomials form a $T$-subalgebra $C(G)$ in $F \langle X \rangle$. Over a field of characteristic $p > 2$ there are algebras $G$ whose algebras of all central polynomials $C (G)$ are not finitely generated as \textit{$T$-subspaces} in $F \langle X \rangle$. However, no example of an algebra $G$ such that $C(G)$ is not finitely generated as a \textit{$T$-subalgebra} is known yet.

In the present paper we construct the first example of a $2$-graded unital associative algebra $B$ over a field of characteristic $p>2$ whose algebra $C_2 (B)$ of all $2$-graded central polynomials is not finitely generated as a $T_2$-subalgebra in the free $2$-graded unital associative $F$-algebra $F \langle Y,Z \rangle$. Here $Y = \{ y_1, y_2, \dots \}$ and $Z = \{ z_1, z_2, \dots \}$ are sets of even and odd free generators of $F \langle Y,Z \rangle$, respectively. We hope that our example will help to construct an algebra $G$ whose algebra $C(G)$ of (ordinary) central polynomials is not finitely generated as a $T$-subalgebra in $F \langle X \rangle$.
\end{abstract}



\maketitle

\section{Introduction}

Let $F$ be a field and let $F \langle X \rangle$ be the free unital associative $F$-algebra on the free generating set $X = \{ x_1, x_2, \dots \}$. Recall that a two-sided ideal $I$ in $F \langle X \rangle$ is called a \textit{$T$-ideal} if $\phi (I) \subseteq I$ for all endomorphisms $\phi$ of $F \langle X \rangle$. Similarly, a subalgebra (a vector subspace) $U$ in $F \langle X \rangle$ is called a \textit{$T$-subalgebra} (a \textit{$T$-subspace}) if $\phi (U) \subseteq U$ for all endomorphisms $\phi$ of $F \langle X \rangle$.

Let $G$ be a unital associative algebra over $F$. Recall that a polynomial $f(x_1,\ldots ,x_n) \in F \langle X \rangle$ is called a \emph{polynomial identity} in $G$ if $f(g_1, \ldots , g_n) = 0$ for all $g_1, \dots , g_n \in G$. One can easily check that, for a given algebra $G$, its polynomial identities form a T-ideal $T(G)$ in $F \langle X \rangle$. The converse also holds: for every T-ideal $I$ in $F \langle X \rangle$ there is an algebra $G$ such that $I = T(G)$, that is, $I$ is the ideal of all polynomial identities satisfied in $G$.

A polynomial $f(x_1,\ldots,x_n)\in F \langle X \rangle$ is \textit{a central polynomial} of $G$ if, for all $g_1, \dots , g_n \in G$, $f(g_1,\dots, g_n)$ is central in $G$. Clearly, $f = f(x_1, \dots, x_n)$ is a central polynomial of $G$ if and only if $[f, x_{n+1}]$ is a polynomial identity of $G$. For a given algebra $G$ its central polynomials form a $T$-subalgebra $C(G)$ in $F \langle X \rangle$. However, not every $T$-subalgebra in $F \langle X \rangle $ coincides with the $T$-subalgebra $C(G)$ of all central polynomials of any algebra $G$.

Let $I$ be a $T$-ideal in $F \langle X \rangle$. A subset $S \subset I$ \textit{generates $I$ as a $T$-ideal} if $I$ is the minimal $T$-ideal  in $F \langle X \rangle$ containing $S$. The $T$-subalgebra and the $T$-subspace of $F \langle X \rangle$ generated by $S$ (as a $T$-subalgebra and a $T$-subspace, respectively) are defined in a similar way. Clearly, the $T$-ideal ($T$-subalgebra, $T$-subspace) generated by $S$ is the ideal (the subalgebra, the vector subspace) in $F \langle X \rangle$ generated by all polynomials $f(a_1,\ldots, a_m)$, where $f=f(x_1, \ldots , x_m) \in S$ and $a_i\in F \langle X \rangle$ for all $i$.

We refer to \cite{drbook, gz, k-brbook, rowenbook} for further terminology and basic results concerning $T$-ideals and algebras with polynomial identities and to \cite{BOR10, BKKS, GKS12, GKS14, GrishinTsybulya09, k-brbook} for an account of results concerning $T$-subspaces and $T$-subalgebras.

Let $F$ be a field of characteristic $0$. Then for each associative $F$-algebra $G$ (unital or not) its ideal of polynomial identities $T(G)$ is a finitely generated $T$-ideal and its subalgebra of central polynomials $C(G)$ is a finitely generated $T$-subspace (and thus a finitely generated $T$-subalgebra). This is because, by Kemer's solution of the Specht problem \cite{kemerbook},  over a field $F$ of characteristic $0$ each $T$-ideal in $F \langle X \rangle$ is finitely generated (as such). Moreover, over such a field $F$ each $T$-subspace (and, therefore, each $T$-subalgebra) in $F \langle X \rangle$ is finitely generated; this has been proved more recently by Shchigolev \cite{Shchigolev01}.

On the other hand, over a field $F$ of characteristic $p>0$ there are associative algebras $G$ such that their ideals of polynomial identities  $T(G)$   are not finitely generated as $T$-ideals in $F \langle X \rangle$. This has been proved by Belov \cite{Belov99}, Grishin \cite{Grishin99} and Shchigolev \cite{Shchigolev99} (see also \cite{Belov00, Grishin00, k-brbook}).

Over a field $F$ of characteristic $p>2$ there are also associative algebras $G$ such that their subalgebras $C(G)$ of central polynomials are not finitely generated as \textit{$T$-subspaces} in $F \langle X \rangle$. In fact, the infinite dimensional Grassmann algebra $E$ over an infinite field $F$ of characteristic $p>2$ is such an algebra: its vector space $C(E)$ of central polynomials is a non-finitely generated \textit{$T$-subspace} in $F \langle X \rangle$ (see \cite{BOR10,BKKS,Grishin10}). However, $C(E)$ is finitely generated as a \textit{$T$-subalgebra} in $F \langle X \rangle$. To the best of our knowledge the following problem is still open.

\begin{problem}
\label{problem1}
Let $F$ be a field of characteristic $p>0$. Find an associative (unital) $F$-algebra $B$ such that its subalgebra of central polynomials $C(B)$ is not finitely generated as a $T$-subalgebra in $F \langle X \rangle$.
\end{problem}

Note that over an infinite field of characteristic $p>2$ many $T$-subalgebras in $F \langle X \rangle$ are known to be non-finitely generated, see \cite{GKS14,Shchigolev00}. Moreover, such non-finitely generated $T$-subalgebras exist in $F \langle x_1, \ldots , x_n \rangle$, where $n>1$ (see \cite{GKS14,Shchigolev00}). However, these non-finitely generated $T$-subalgebras do not coincide with the subalgebra $C(G)$ of all central polynomials of any algebra $G$.

It is worth to mention that if $R$ is a Noetherian unital associative and commutative ring then each $T$-ideal in $R \langle x_1, \ldots , x_n \rangle$ $(n \ge 1)$ is finitely generated; this has been proved recently by Belov \cite{Belov10}.

\medskip
Recall that if $H$ is an additive group and $G$ is an $F$-algebra then $G$ is \textit{$H$-graded} if $G = \oplus_{h\in H} G_h$ where $G_h$ are vector subspaces of $G$ and $G_h G_{h'} \subseteq G_{h+h'}$ for every $h, h' \in H$. Note that $G_0$ is a subalgebra of $G$. In this paper, unless otherwise stated, we fix $H = \mathbb Z / 2 \, \mathbb Z$ so $G = G_0 \oplus G_1$. We refer to the elements of $G_0$ as even ones and to those of $G_1$ as odd ones; the adjective \textit{$2$-graded} will stand for $(\mathbb Z / 2 \, \mathbb Z)$-graded.

Let $Y = \{ y_1, y_2, \dots \}$, $Z= \{ z_1, z_2, \dots \}$. Let $A = F \langle Y, Z \rangle$ be the free unital associative algebra over $F$ with a free generating set $Y \cup Z$. Define a $2$-grading on $A$ by setting $y_i \in A_0$, $z_j \in A_1$ for all $i,j$. It is clear that $A_0$ is the linear span of all monomials in variables $y_i, z_j$ that contain even number of variables $z_j \in Z$ and $A_1$ is spanned by the monomials that contain odd number of variables $z_j$. We have $A = A_0 \oplus A_1$; $A_0 A_0, A_1 A_1 \subseteq A_0$; $A_1 A_0, A_0 A_1 \subseteq A_1$.

A two-sided ideal $I$ in $A$ is called a \textit{$T_2$-ideal} if $\phi (I) \subseteq I$ for all $2$-graded endomorphisms $\phi$ of $A$, that is, for all endomorphisms $\phi$ such that $\phi (A_0) \subseteq A_0$, $\phi (A_1) \subseteq A_1$. Similarly, a subalgebra (a vector subspace) $U$ in $A$ is called a \textit{$T_2$-subalgebra} (a \textit{$T_2$-subspace}) if $\phi (U) \subseteq U$ for all $2$-graded endomorphisms $\phi$ of $A$.

Let $G = G_0 \oplus G_1$ be a $2$-graded unital associative algebra over $F$. Recall that a polynomial $f(y_1, y_2, \dots ; z_1, z_2, \dots ) \in A$ is called a \textit{$2$-graded polynomial identity} in $G$ if $f(g_1, g_2, \dots ; g_1', g_2', \dots ) = 0$ for all $g_1, g_2, \dots  \in G_0$, $g_1', g_2' \dots \in G_1$. One can easily check that, for a given $2$-graded algebra $G$, its $2$-graded polynomial identities form a $T_2$-ideal $T_2(G)$ in $A$. The converse also holds: for every $T_2$-ideal $I$ in $A$ there is a $2$-graded algebra $G$ such that $I = T_2(G)$, that is, $I$ is the ideal of all $2$-graded polynomial identities satisfied in $G$.

A polynomial $f(y_1, y_2, \dots ; z_1, z_2, \dots ) \in A$ is \textit{a $2$-graded central polynomial} of $G$ if, for all $g_1, g_2, \dots \in G_0$ and all $g_1', g_2', \dots \in G_1$, $f(g_1,g_2, \dots; g_1', g_2', \dots )$ is central in $G$. For a given $2$-graded algebra $G$ its $2$-graded central polynomials form a $T_2$-subalgebra $C_2(G)$ in $A$. However, not every $T_2$-subalgebra in $A$ coincides with the $T_2$-subalgebra $C_2(G)$ of all $2$-graded central polynomials of any algebra $G$.

Let $I$ be a $T_2$-ideal in $A$. A subset $S \subset I$ \textit{generates $I$ as a $T_2$-ideal} if $I$ is the minimal $T_2$-ideal  in $A$ containing $S$. A $T_2$-subalgebra and a $T_2$-subspace of $A$ generated by $S$ (as a $T_2$-subalgebra and a $T_2$-subspace, respectively) are defined in a similar way.

Graded identities is a powerful tool for studying PI algebras. They play an essential role in the structure theory of the $T$-ideals developed by Kemer, see \cite{kemerbook}. Soon after Kemer's achievments  graded identities became object of extensive studies. We refer to \cite{gz} for further terminology, basic results and reference concerning $T_2$-ideals, graded polynomial identities and graded central polynomials.

The aim of our paper is to solve the following (simpler) graded analog of Problem \ref{problem1}.

\begin{problem}
Let $F$ be a field of characteristic $p>0$. Find a $2$-graded associative (unital) $F$-algebra $B$ such that its subalgebra of $2$-graded central polynomials $C_2(B)$ is not finitely generated as a $T_2$-subalgebra in $A$.
\end{problem}

We hope that our example will help to solve Problem \ref{problem1}, that is, to construct an algebra $G$ whose algebra $C(G)$ of (ordinary) central polynomials is not finitely generated as a $T$-subalgebra in $F \langle X \rangle$.

\medskip
Let $T$ be the (two-sided) ideal in $A$ generated by all polynomials $[a_1, a_2, a_3]$ $(a_i \in A)$. Clearly, $T$ is a $T$-ideal and, therefore, a $T_2$-ideal in $A$.

Let $A^{(k)}$ $(k = 0,1,2, \dots )$ be the linear span of all monomials in variables $y_i \in Y, z_j \in Z$ that are of degree $k$ in the variables $z_j$ $(j = 1,2, \dots )$. For example, $y_1 z_2 y_3 z_4 z_5 \in A^{(3)}$. Then $A = \oplus_{i = 0}^{\infty} A^{(i)}$. Define $I_k = \sum_{i \ge k} A^{(i)}$ $(k = 1,2, \dots )$. It is clear that, for each $k$, $I_k$ is a $T_2$-ideal in $A$.

Define $U = (T \cap I_p) + I_{p+1} $. Let $B = A / U$. Since $U$ is a $2$-graded ideal in $A$, the quotient algebra $B$ is a $2$-graded unital associative algebra with the $2$-grading inherited from $A$, $B = B_0 \oplus B_1$, $B_0 = (A_0 + U)/U$, $B_1 = (A_1 + U)/U$. Our main result is as follows.

\begin{theorem}
\label{maintheorem}
Let $F$ be an infinite field of characteristic $p>2$. Then the algebra $C_2 (B)$ of all $2$-graded central polynomials of $B$ is not finitely generated as a $T_2$-subalgebra in $A = F \langle Y, Z \rangle$.
\end{theorem}

The idea of the proof is as follows. We will prove that the image $(C_2 (B) + U)/U$ of the algebra $C_2 (B)$ of the central polynomials of $B$ is not finitely generated as a \textit{$T_2$-subspace} in $A/U$. To prove this we will make use of the description of the central polynomials of the unital infinite-dimensional Grassmann algebra $E$ over an infinite field $F$ of characteristic $p>2$ obtained in \cite{BOR10,BKKS,Grishin10}.

On the other hand, we will check that, up to a scalar term, $(C_2 (B) + U)/U$ is an algebra with null multiplication. It follows that any set containing the unity $1$ that generates $(C_2 (B) + U)/U$ as a \textit{$T_2$-subalgebra} also generates it as a \textit{$T_2$-subspace}. Since $(C_2 (B) + U)/U$ is not finitely generated as a $T_2$-subspace, $(C_2 (B) + U)/U$ is not finitely generated as a $T_2$-subalgebra in $A/U$ as well. It follows that $C_2 (B)$ is not finitely generated as a $T_2$-subalgebra in $A$, as required.

The paper is organized as follows. In Section 2 we state and prove some results about (ordinary) central polynomials of the Grassmann algebra $E$ that we need to prove the main result. In Section 3 we give a proof of Theorem \ref{maintheorem}.


\section{The central polynomials of the Grassmann algebra}

Let $F$ be an infinite field of characteristic $p>2$. Define $X = Y \cup Z$, $x_{2 i -1} = y_i$, $x_{2i} = z_i$ $(i \in \mathbb N)$. Then $X = \{ x_1, x_2, \dots \}$ and $A = F \langle X \rangle$ is the free unital associative $F$-algebra on the free generating set $X$.

Let $E$ be the infinite-dimensional unital Grassmann algebra over $F$. Then $E$ is generated by elements $e_i$ $(i = 1, 2, \dots )$ such that $e_i e_j = - e_j e_i$, $e_i^2 = 0$ for all $i, j$ and the set
\[
\{ e_{i_1} e_{i_2} \dots e_{i_k} \mid k \ge 0, \, i_1 < i_2 < \dots < i_k \}
\]
forms a basis of $E$ over $F$. Let $T(E)$ be the $T$-ideal of all (ordinary) polynomial identities of $E$. Then $T(E) = T$ (see, for instance, \cite{GiambrunoKoshlukov01}).

The $T$-subspace $C = C(E)$ of all (ordinary) central polynomials of $E$ was described in \cite{BOR10,BKKS,Grishin10}. Let $q(x_1, x_2) = x_1^{p-1} [x_1, x_2] x_2^{p-1}$ and let, for each $n \ge 1$,
\[
q_n = q_n (x_1, \dots , x_{2n}) = q(x_1, x_2) q(x_3, x_4) \dots q(x_{2n-1}, x_{2n}).
\]
The $T$-subspace $C(E)$ is generated (as a $T$-subspace in $A$) by the polynomial $x_1[x_2, x_3, x_4]$ together with the polynomials $x_0^p, x_0^p q_1, x_0^p q_2, \dots , x_0^p q_n, \dots $ (see \cite{BOR10,BKKS,Grishin10}).

Let $M \subset A$ be the set of monic (non-commutative) monomials in $x_i$ $(i \in \mathbb N)$,
\[
M = \{ x_{i_1} x_{i_2} \dots x_{i_{\ell}} \mid l \ge 0, i_s \in \mathbb N \mbox{ for all } s \} .
\]
The following lemma can be deduced, for instance, from \cite[Proof of Theorem 2]{BKKS}.

\begin{lemma}
\label{C}
The vector subspace $C$ is spanned (as a vector space over $F$) by all polynomials $g_1 [g_2, g_3, g_4]$ and $[g_1, g_2]$ $(g_i \in M)$ together with the polynomials
\begin{equation}
\label{product}
x_{i_1}^{p m_1} x_{i_2}^{p m_2} \dots \, x_{i_k}^{p m_k} \ x_{j_1}^{p-1} [x_{j_1}, x_{j_2}] x_{j_2}^{p-1} x_{j_3}^{p-1} [x_{j_3}, x_{j_4}] x_{j_4}^{p-1} \dots \, x_{j_{2 \ell - 1}}^{p-1} [x_{j_{2 \ell -1}}, x_{j_{2 \ell}}] x_{j_{2 \ell }}^{p-1}
\end{equation}
where $k, \ell \ge 0$, $i_1 < i_2  \dots < i_k$, $j_1 < j_2 < \dots < j_{2 \ell}$, $m_i > 0$ for all $i$.
\end{lemma}

\begin{proof}[Sketch of proof.] It is clear that all polynomials $g_1[g_2,g_3,g_4]$ and $[g_1,g_2]$ $(g_i \in M)$ belong to $C$; it is well known and straightforward to check that the polynomials of the form (\ref{product}) also belong to $C$. Thus, to prove Lemma \ref{C} it suffices to check that each polynomial $f \in C$ belongs to the linear span of the polynomials $g_1[g_2,g_3,g_4]$ and $[g_1,g_2]$ $(g_i \in M)$ and the polynomials of the form (\ref{product}).

It is well known (see, for example, \cite[Proposition 9]{BKKS}) that  the vector space $A/T$ over $F$ has a basis formed by the elements
\[
x_{i_1}^{n_1} x_{i_2}^{n_2} \dots \, x_{i_k}^{n_k} \ [x_{j_1}, x_{j_2}] [x_{j_3}, x_{j_4}] \dots \, [x_{j_{2 \ell -1}}, x_{j_{2 \ell}}]  + T
\]
where $k, \ell \ge 0$, $i_1 < i_2 , \dots < i_k$, $j_1 < j_2 < \dots < j_{2 \ell}$, $n_i > 0$ for all $i$.

Let $f \in C$ be an arbitrary element of $C$. Since the field $F$ is infinite, the $T$-subspace $C$ is spanned by multi-homogeneous polynomials so we may assume without loss of generality that $f$ is multi-homogeneous. For all $i$, let $d_i$ be the degree of $f$ with respect to $x_i$, $d_i = \deg_{x_i} f$. If, for some $i$, $p$ does not divide $d_i$ then, by \cite[Lemma 12]{BKKS}, $f +T$ belongs to the vector space of $A/T$ spanned by the polynomials $[g_1,g_2] +T$ where $g_j \in A$ or, equivalently, where $g_j \in M$. If, on the other hand, $p$ divides $d_i$ for all $i$ then one can check that $f + T = g +T$ for some linear combination $g$ of polynomials of the form (\ref{product}). It follows that $C/T$ is spanned by the polynomials $[g_1,g_2] + T$ $(g_j \in M)$ and the polynomials $h + T$ where $h$ is of the form (\ref{product}). Since $T$ is spanned by the polynomials $g_1 [g_2, g_3, g_4]$ $(g_j \in M)$, the result follows. See \cite[Proof of Theorem 2]{BKKS} for details.
\end{proof}

Let $D_{n}$ be the vector subspace of $A$ generated by all polynomials $g_1 [g_2, g_3, g_4]$ and $[g_1, g_2]$ $(g_i \in A)$ together with all polynomials
\begin{equation}
\label{product_Dn}
g_{1}^{p m_1} g_{2}^{p m_2} \dots \, g_{k}^{p m_k} \ h_{1}^{p-1} [h_{1}, h_{2}] h_{2}^{p-1} h_{3}^{p-1} [h_{3}, h_{4}] h_{4}^{p-1} \dots \, h_{{2 \ell - 1}}^{p-1} [h_{{2 \ell -1}}, h_{{2 \ell}}] h_{{2 \ell }}^{p-1} \qquad (g_i, h_j \in A)
\end{equation}
such that $k \ge 0$, $0 \le \ell \le n$. It is clear that, for each $n \ge 0$,  $D_n$ is a $T$-subspace in $A$. Note that, for each $n$, $D_n \subset C$. Indeed, each polynomial (\ref{product_Dn}) is a homomorphic image of a polynomial (\ref{product}). By Lemma \ref{C}, each polynomial (\ref{product}) belongs to $C$; since $C$ is a $T$-subspace, all homomorphic images of polynomials (\ref{product}) also belong to $C$. Hence, all polynomials (\ref{product_Dn}) belong to $C$. It follows that, for each $n$, $D_n \subset C$, as claimed. On the other hand, it is clear that $C \subseteq \bigcup_{n \ge 0} D_n$ and, therefore,  $C = \bigcup_{n \ge 0} D_n$.

The following lemma is an immediate corollary of Shchigolev's result \cite[Lemma 13]{Shchigolev00} (see also \cite[Proposition 13]{BKKS}). It is worth to mention that this result of \cite{Shchigolev00} has been used in \cite{Belov99, Shchigolev99} (see also \cite{Belov00, k-brbook}) to construct the first examples of non-finitely generated $T$-ideals in $F \langle X \rangle$ over a field of characteristic $p>2$.

\begin{lemma}
\label{q_n+1_D_n}
For each $n \ge 0$, $q_{n+1} \notin D_n$.
\end{lemma}

Note that from the statement of \cite[Lemma 13]{Shchigolev00} one can deduce only a weaker assertion: for each $n \ge 0$, there exists $k(n) > n$ such that $q_{k(n)} \notin D_n$. However, it follows from the proof of \cite[Lemma 13]{Shchigolev00} that one can choose $k(n) = n + 1$.

Since $D_0 \subset D_1 \subset \dots \subset D_n \subset \dots $, Lemma \ref{q_n+1_D_n} implies he following.
\begin{corollary}
\label{q_k_D_n}
For each $n \ge 0$ and each $k >n$, $q_k \notin D_n$.
\end{corollary}

Let $S^{(n)}$ be the set of all polynomials
\begin{equation*}
\label{product2}
x_{i_1}^{p m_1} x_{i_2}^{p m_2} \dots \, x_{i_k}^{p m_k} \ x_{j_1}^{p-1} [x_{j_1}, x_{j_2}] x_{j_2}^{p-1} x_{j_3}^{p-1} [x_{j_3}, x_{j_4}] x_{j_4}^{p-1} \dots \, x_{j_{2 \ell - 1}}^{p-1} [x_{j_{2 \ell -1}}, x_{j_{2 \ell}}] x_{j_{2 \ell }}^{p-1}
\end{equation*}
such that $ 0 \le \ell \le n$, $ k \ge 0$, $i_1 < i_2 , \dots < i_k$, $j_1 < j_2 < \dots < j_{2 \ell}$, $m_i > 0$ for all $i$. The following lemma follows immediately from \cite[Theorem 2.1]{GrishinTsybulya09}. However, we will deduce it here from Corollary \ref{q_k_D_n} in order to have the paper more self-contained.

\begin{lemma}
\label{D_n}
For each $n \ge 0$, the vector subspace $D_n$ is spanned (as a vector space over $F$) by the set $S^{(n)}$ together with all polynomials $g_1 [g_2, g_3, g_4]$ and $[g_1, g_2]$ $(g_i \in M)$.
\end{lemma}

\begin{proof}
It is clear that all polynomials of $S^{(n)}$ and all polynomials $g_1 [g_2, g_3, g_4]$ and $[g_1, g_2]$ $(g_i \in M)$ belong to $D_n$. Therefore, it suffices to check that each polynomial $f \in D_n$ can be written as a linear combination of these polynomials.

Suppose, in order to get a contradiction, that $f \in D_n$ can not be written as a linear combination of elements of $S^{(n)}$ and polynomials of the forms $g_1 [g_2, g_3, g_4]$ and $[g_1, g_2]$ $(g_i \in M)$. Since the field $F$ is infinite, we may assume without loss of generality that $f$ is multi-homogeneous. By Lemma \ref{C}, $f = f_1 + f_2 + f_3$ where $f_1$ is a linear combination of polynomials of the forms $g_1 [g_2, g_3, g_4]$ and $[g_1, g_2]$ $(g_i \in M)$, $f_2$ is a linear combination of polynomials of the form (\ref{product}) with $\ell \le n$ and $f_3$ is a linear combination of polynomials of the form (\ref{product}) with $\ell > n$. Since $f_1, f_2 \in D_n$, we may assume that $f = f_3$. Hence, $f = \sum_{t = 1}^s  \alpha_t h_t$ where $\alpha_t \in F \setminus \{ 0 \}$ for all $t$ and each $h_t$ is a polynomial of the form (\ref{product}) with $\ell > n$, that is,
\[
h_t  = x_{i_{t1}}^{p m_{t1}} \dots \, x_{i_{tk(t)}}^{p m_{tk(t)}} \ x_{j_{t1}}^{p-1} [x_{j_{t1}}, x_{j_{t2}}] x_{j_{t2}}^{p-1}\dots \, x_{j_{t(2 \ell (t) - 1)}}^{p-1} [x_{j_{t (2 \ell (t) -1)}}, x_{j_{t (2 \ell (t))}}] x_{j_{t (2 \ell (t))}}^{p-1}
\]
where $\ell (t) > n$ for all $t$.

Suppose (renumerating the terms $h_t$ if necessary) that $\ell (1) \le \ell (t)$ for all $t$. Let $\phi$ be the endomorphism of $A$ such that $\phi (x_{j_{1r}}) = x_r$ for $r = 1, \dots , 2 \ell (1)$ and $\phi (x_q) = 1$ for all other $x_q$. Then
\[
\phi (h_1) = x_{1}^{p m_{1}} \dots \, x_{2 \ell (1)}^{p m_{2 \ell (1)}} \ x_1^{p-1} [x_1, x_2] x_2^{p-1}\dots \, x_{2 \ell (1) - 1}^{p-1} [x_{2 \ell (1) -1}, x_{2 \ell (1)}] x_{2 \ell (1)}^{p-1}
\]
for some $m_i \ge 0$ $(i = 1, \dots , 2 \ell (1))$. On the other hand, $\phi (h_t) = 0$ for all $t >1$ because, for each $t >1$, there is $j_{t q}$ such that $\phi (x_{j_{t q}}) = 1$ and, therefore,
\[
\phi \big( x_{j_{t1}}^{p-1} [x_{j_{t1}}, x_{j_{t2}}] x_{j_{t2}}^{p-1}\dots \, x_{j_{t(2 \ell (t) - 1)}}^{p-1} [x_{j_{t (2 \ell (t) -1)}}, x_{j_{t (2 \ell (t))}}] x_{j_{t (2 \ell (t))}}^{p-1} \big) = 0.
\]
Thus, $\phi (f) = \alpha_1 h_1$. Since $D_n$ is a $T$-subspace in $A$ and $\alpha_1 \ne 0$, we have $h_1 \in D_n$.

Let $\psi$ be the automorphism of $A$ such that $\psi (x_i ) = x_i + 1$ for all $i$. Then $\psi (h_1) \in D_n$. One can check that
\begin{multline*}
\psi (h_1) + T = ( x_{1} ^{p } + 1)^{m_{1}} \dots \, (x_{2 \ell (1)}^{p } +1)^{m_{2 \ell (1)}}
\\
\times (x_1 + 1)^{p-1} [x_1, x_2] (x_2 +1)^{p-1}\dots \, (x_{2 \ell (1) - 1} + 1)^{p-1} [x_{2 \ell (1) -1}, x_{2 \ell (1)}] (x_{2 \ell (1)} + 1)^{p-1} + T.
\end{multline*}
Note that the multi-homogeneous component $h' + T$ of $\psi (h_1) + T$ of degree $p$ in all variables $x_1, \dots , x_{2 \ell (1)}$ coincides with $q_{\ell (1) } + T$,
\[
h' + T = x_1^{p-1} [x_1, x_2] x_2^{p-1}\dots \, x_{2 \ell (1) - 1}^{p-1} [x_{2 \ell (1) -1}, x_{2 \ell (1)}] x_{2 \ell (1)}^{p-1} + T = q_{\ell (1)} + T.
\]
Since $\psi (h_1) + T \in D_n /T$, we have $h' + T \in D_n/T$, that is, $q_{\ell (1)} + T \in D_n/T$ so $q_{\ell (1)} \in D_n$. This contradicts Corollary \ref{q_k_D_n} because $\ell (1) > n$. The result follows.
\end{proof}

\section{Proof of Theorem \ref{maintheorem}}

Let $1 \cdot F$ denote the linear span of unity $1 \in A$.

\begin{lemma}
\label{C_2(B)}
The vector space $C_2 (B)$ is a direct sum of the vector spaces $1 \cdot F$, $C \cap A^{(p)}$ and $I_{p+1}$,
\begin{equation}
\label{C2B}
C_2 (B) = 1 \cdot F \oplus ( C \cap A^{(p)})  \oplus I_{p+1}.
\end{equation}
\end{lemma}

\begin{proof}
Since $1 \cdot F + A^{(p)} + I_{p+1} = 1 \cdot F \oplus A^{(p)} \oplus I_{p+1}$, it suffices to prove that $C_2 (B) = 1 \cdot F + ( C \cap A^{(p)})  + I_{p+1}$.

Suppose that $f \in 1 \cdot F + ( C \cap A^{(p)})  +  I_{p+1}$. Since the algebra $B$ is generated by the elements $y_i +U$, $z_i +U$ $(i \in \mathbb N)$, to prove that $f \in C_2 (B)$ it suffices to check that $[f, y_i], [f,z_i] \in U$ for all $i \in \mathbb N$. Since $f \in I_p$, we have $[f, z_i] \in I_{p+1} \subset U$ for all $i$. Hence, it remains to check that $[f, y_i] \in U$ for all $i$.

Let $f = f^{(0)} + f^{(1)} + f^{(2)}$, where $f^{(0)} \in F$, $f^{(1)} \in (C \cap A^{(p)})$ and $f^{(2)} \in I_{p+1}$; then $[f, y_i] = [f^{(1)}, y_i] + [f^{(2)}, y_i]$. Since $f^{(2)} \in I_{p+1}$, we have $[f^{(2)} , y_i] \in I_{p+1} \subset U$. On the other hand, $f^{(1)} \in C$ so $[f^{(1)}, y_i] \in T$. Since $f^{(1)} \in A^{(p)}$, we have $[f^{(1)}, y_i] \in A^{(p)}$ so $[f^{(1)}, y_i] \in (T \cap A^{(p)}) \subset U$. Hence, $[f, y_i] \in U$ for each $i$.

Thus, if $f \in 1 \cdot F + ( C \cap A^{(p)})  +  I_{p+1}$ then $f \in C_2 (B)$, that is, $1 \cdot F + ( C \cap A^{(p)})  +  I_{p+1} \subseteq C_2 (B)$.

\medskip
Now suppose that $f \in C_2 (B)$, that is, $[f, y_i], [f,z_i] \in U$ for all $i \in \mathbb N$.

Let $f = f_0 + f_1 + \dots + f_p + f_{p+1}$, where $f_j \in A^{(j)}$ $(j = 0, 1, \dots , p)$, $f_{p+1} \in I_{p+1}$. Then $[f, y_i] = [f_0, y_i] + \dots + [f_p, y_i] + [f_{p+1}, y_i] \in U$. Since $U \subset I_p$ and $[f_{\ell}, y_i] \in A^{(\ell )}$ $( \ell = 0, 1, \dots , p)$, we have $[f_{\ell}, y_i] = 0$ for all $i \in \mathbb N$ and all $\ell$, $0 \le \ell <p$. It is clear that if $g \in A$ and $[g, y_i] = 0$ for all $i \in \mathbb N$ then $g \in 1 \cdot F$; hence, $f_0 \in 1 \cdot F$ and $f_{\ell}=0$ if $0 < \ell < p$, that is, $f = f_0 + f_p +f_{p+1}$, where $f_0 \in 1 \cdot F$, $f_p \in A^{(p)}$ and $f_{p+1} \in I_{p+1}$. It follows that to prove that $f \in 1 \cdot F + ( C \cap A^{(p)})  + I_{p+1}$ it suffices to check that $f_p \in C $.

Let $g = g(x_1, \dots , x_k) \in A$. We claim that to check that $g \in C$ it suffices to check that $[g, x_j] \in T$ for some $j>k$. Indeed, $g \in C$ if and only if $[g, x_i] \in T$ for all $i$. If $[g, x_j]  \in T$ then $\psi ([g, x_j]) = [\psi (g), \psi (x_j)] \in T$ for each endomorphism $\psi$ of $A$ because $T$ is a $T$-ideal in $A$. For any $i$, take $\psi$ such that $\psi (x_{\ell}) = x_{\ell}$ for all $\ell = 1, 2, \dots , k$ and $\psi (x_j) = x_i$; then $[g, x_i] = [\psi (g(x_1, \dots , x_k), \psi (x_j)] \in T$ so $g \in C$, as claimed.

Now let $f_p = f_p(y_1, \dots ,y_k; z_1, \dots z_k)$. Take $j >k$. Since $f \in C_2 (B)$, we have
\[
[f, y_j] =[f_p, y_j] + [f_{p+1}, y_j] \in U.
\]
Since $f_{p+1} \in I_{p+1}$, we have $[f_{p+1}, y_j] \in I_{p+1} \subset U$ and therefore
\[
[f_p, y_j] \in U = (T \cap A^{(p)}) \oplus I_{p+1} \subset A^{(p)} \oplus I_{p+1} .
\]
Since $[f_p, y_j] \in  A^{(p)}$, we have $[f_p, y_j] \in (T \cap A^{(p)}) \subset T$. By the observation made in the previous paragraph this implies that $[f_p, x_i] \in T$ for all free generators $x_i$ of $A$, that is, $f_p \in C$. It follows that $f \in 1 \cdot F + ( C \cap A^{(p)})  + I_{p+1}$ and, therefore, $C_2 (B) \subseteq 1 \cdot F + ( C \cap A^{(p)})  + I_{p+1}$.

This completes the proof of Lemma \ref{C_2(B)}.
\end{proof}

Let $W_n = 1 \cdot F + (D_n \cap I_p) + I_{p+1}$ $(n \ge 0)$. Since $D_n$ is a $T$-subspace (and therefore a $T_2$-subspace) in $A$ and $I_p, I_{p+1}$ are $T_2$-ideals (and thus $T_2$-subspaces), $W_n$ is a $T_2$-subspace in $A$. On the other hand, $W_n$ is a subalgebra in $A$ because $\big( (D_n \cap I_p) + I_{p+1} \big) \cdot \big( (D_n \cap I_p) + I_{p+1} \big) \subset I_{p+1}$ so $W_n \cdot W_n = W_n$. Hence, $W_n$ is a $T_2$-subalgebra in $A$.

\begin{lemma}
For each $n \ge 0$, the vector subspace $W_n$ of $A$ is a direct sum of the vector subspaces $1 \cdot F$, $D_n \cap A^{(p)}$  and $I_{p+1}$,
\begin{equation}
\label{Wk}
 W_n = 1 \cdot F \oplus (D_n \cap A^{(p)}) \oplus I_{p+1}.
\end{equation}
\end{lemma}

\begin{proof}
Note that $D_n$ is spanned over $F$ by all polynomials (\ref{product2}) together with all polynomials $g_1 [g_2, g_3, g_4]$ and $[g_1, g_2]$, where all $g_i \in M$ are monic monomials in $x_i$ $(i \in \mathbb N)$. Since each of these polynomials belongs to $A^{(s)}$ for a suitable $s \in \mathbb N$, we have
\[
D_n = (D_n \cap A^{(0)}) \oplus (D_n \cap A^{(1)}) \oplus \dots \oplus (D_n \cap A^{( \ell )}) \oplus \dots
\]
It follows that $D_n \cap I_p = (D_n \cap A^{(p)}) \oplus \dots \oplus (D_n \cap A^{( \ell )}) \oplus \dots $ so  $(D_n \cap I_p) + I_{p+1} = (D_n \cap A^{(p)}) \oplus I_{p+1}$. Thus,
\[
W_n = 1 \cdot F + (D_n \cap I_p) + I_{p+1} = 1 \cdot F \oplus (D_n \cap A^{(p)}) \oplus I_{p+1},
\]
as required.
\end{proof}

Now we are in a position to complete the proof of Theorem \ref{maintheorem}. Since $C = \bigcup_{n=0}^{\infty} D_n$, we have $C \cap A^{(p)} = \bigcup_{n=0}^{\infty} (D_n \cap A^{(p)})$ so, by (\ref{C2B}) and (\ref{Wk}),
\begin{equation}
\label{union}
C_2 (B) = \bigcup_{n=0}^{\infty} W_n.
\end{equation}
Note that $D_0 \cap A^{(p)} \varsubsetneqq D_1 \cap A^{(p)} \varsubsetneqq \dots \varsubsetneqq D_n \cap A^{(p)} \varsubsetneqq \dots $ because, by Lemma \ref{q_n+1_D_n},
\[
q_n (z_1,y_2, \dots , y_{2n-1}, y_{2n}) = z_{1}^{p-1} [z_{1}, y_{2}] y_{2}^{p-1} y_{3}^{p-1} [y_{3}, y_{4}] y_{4}^{p-1} \dots \, y_{2 n - 1}^{p-1} [y_{2 n -1}, y_{2 n}] y_{2 n}^{p-1}
\]
belongs to $(D_n \cap A^{(p)}) \setminus (D_{n-1} \cap A^{(p)})$. Hence,
\begin{equation}
\label{chain}
W_0 \varsubsetneqq W_1 \varsubsetneqq \dots \varsubsetneqq W_n \varsubsetneqq \dots .
\end{equation}
By (\ref{union}) and (\ref{chain}), the $T_2$-subalgebra $C_2 (B)$ is not finitely generated (as a $T_2$-subalgebra in $A$). This completes the proof of Theorem \ref{maintheorem}.

\textbf{Remark.} Theorem \ref{maintheorem} and most of its proof remain valid if $F$ is a finite field of characteristic $p>2$. In this case the $T$-subspace $C$ of $A$ in Section 2 should be defined by $C = C(A/T)$. Note that for a finite field $F$ we have $T \ne T(E)$ and $C \ne C(E)$, see \cite{BOR11}.

\section*{Acknowledgements}

The second author was supported by CNPq grants 307328/2012-0 and 480139/2012-1, by DPP/UnB and by CNPq-FAPDF PRONEX grant 2009/00091-0(193.000.580/2009)


\begin{thebibliography}{10}

\bibitem{BOR10}
C. Bekh-Ochir, S.A. Rankin, \textit{The central polynomials of the infinite dimensional unitary and nonunitary Grassmann algebras}, J. Algebra Appl. 9 (2010), 687--704.

\bibitem{BOR11}
C. Bekh-Ochir, S.A. Rankin, \textit{The identities and the central polynomials of the infinite dimensional unitary Grassmann algebra over a finite field}, Comm. Algebra \textbf{39} (2011), 819--829.

\bibitem{Belov99}
A.Ya. Belov, \textit{On non-Specht varieties} (Russian), Fundam. Prikl. Mat. \textbf{5} (1999), 47--66.

\bibitem{Belov00}
A.Ya. Belov, \textit{Counterexamples to the Specht problem}, Sb.: Math. \textbf{191} (2000), 329--340.

\bibitem{Belov10}
A.Ya. Belov, \textit{The local finite basis property and the local representability of varieties of associative rings}, Izv.: Math. \textbf{74} (2010), 1--126.

\bibitem{BKKS}
A. Brand\~ao Jr., P. Koshlukov, A. Krasilnikov, E.A. Silva, \textit{The central polynomials for the Grassmann algebra,} Israel J. Math. \textbf{179} (2010), 127-144.

\bibitem{drbook}
V. Drensky, \textit{Free algebras and PI-algebras}. Graduate course in  algebra, Springer, Singapore, 1999.

\bibitem{drformanekbook}
V. Drensky, E. Formanek, \textit{Polynomial identity rings}. Advanced Courses in Mathematics. CRM Barcelona. Birkh\"auser Verlag, Basel, 2004.

\bibitem{GiambrunoKoshlukov01} A. Giambruno, P. Koshlukov, \textit{On the identities of the Grassmann algebras in characteristic $p>0$}, Israel J. Math. \textbf{122} (2001), 305--316.

\bibitem{gz}
A. Giambruno, M. Zaicev, \textit{Polynomial identities and asymptotic methods}. Mathematical Surveys and Monographs, \textbf{122}. American Mathematical Society, Providence, RI, 2005.

\bibitem{GKS12}
D.J. Gon\c{c}alves, A. Krasilnikov, I. Sviridova, \textit{Limit $T$-subspaces and the central polynomials in $n$ variables of the Grassmann algebra}, J. Algebra \textbf{371} (2012), 156--174.

\bibitem{GKS14}
D.J. Gon\c{c}alves, A. Krasilnikov, I. Sviridova, \textit{Limit $T$-subalgebras in free associative algebras}, J. Algebra \textbf{412} (2014), 264--280.

\bibitem{Grishin99}
A.V. Grishin, \textit{Examples of $T$-spaces and $T$-ideals of characteristic $2$ without the finite basis property}, (Russian), Fundam. Prikl. Mat.  \textbf{5}  (1999),  101--118.

\bibitem{Grishin00}
A.V. Grishin, \textit{On non-Spechtianness of the variety of associative rings that satisfy the identity $x\sp {32}=0$}, Electron. Res. Announc. Amer. Math. Soc.  \textbf{6}  (2000), 50--51 (electronic).

\bibitem{Grishin10}
A.V. Grishin, \textit{On the structure of the centre of a relatively free Grassmann algebra}, Russ. Math. Surv. \textbf{65} (2010), 781--782.

\bibitem{GrishinTsybulya09}
A.V. Grishin, L.M. Tsybulya, \textit{On the multiplicative and T-space structure of the relatively free Grassmann algebra}, Sb.: Math.  \textbf{200} (2009), 1299--1338.

\bibitem{k-brbook} A. Kanel-Belov, L.H. Rowen, \textit{Computational aspects of polynomial identities}. Research
Notes in Mathematics, \textbf{9}. A~K~Peters, Ltd., Wellesley, MA, 2005.

\bibitem{kemerbook}
A.R. Kemer, \textit{Ideal of identities of associative algebras}, Translations of Mathematical Monographs, \textbf{87}. American Mathematical Society, Providence, RI, 1991.

\bibitem{rowenbook}
L.H. Rowen, \textit{Polynomial Identities in Ring Theory}. Pure and Applied Mathematics, \textbf{84}, Acad. Press, New York-London, 1980.

\bibitem{Shchigolev99}
V.V. Shchigolev, \textit{Examples of infinitely based $T$-ideals} (Russian), Fundam. Prikl. Mat. \textbf{5} (1999), 307--312.

\bibitem{Shchigolev00}
V.V. Shchigolev, \textit{Examples of infinitely  basable $T$-spaces} Sb.: Math. \textbf{191} (2000), 459--476.

\bibitem{Shchigolev01}
V.V. Shchigolev, \textit{Finite basis property of $T$-spaces over fields of characteristic zero}, Izv.: Math. \textbf{65} (2001), 1041--1071.

\end{thebibliography}
\end{document}